\documentclass[12pt]{article}
\usepackage[centertags]{amsmath}
\usepackage{amsfonts}
\usepackage{amssymb}
\usepackage{latexsym}
\usepackage{amsthm}
\usepackage{newlfont}
\usepackage{graphicx}
\usepackage{listings}
\usepackage{booktabs}
\usepackage{abstract}
\newlength{\defbaselineskip}
\newcommand{\setlinespacing}[1]%
           {\setlength{\baselineskip}{#1 \defbaselineskip}}

\newcommand{\actaqed}{\hfill $\actabox$}
{\medskip\noindent \textit{Proof of #1. }}%
{\actaqed \medskip}

\def\D{{\mathcal D}}

\def\R{{\mathbb R}}
\def \<{\langle}
\def\>{\rangle}

\def \e{\epsilon}

\def \ff{\varphi}
\def\oo{\omega}

\def \rank{\operatorname{rank}}

\newtheorem{Theorem}{Theorem}[section]
\newtheorem{Lemma}{Lemma}[section]

\newtheorem{Proposition}{Proposition}[section]
\newtheorem{Remark}{Remark}[section]

\newtheorem{Corollary}{Corollary}[section]
\numberwithin{equation}{section}

\begin{document}
\title{{A remark on covering} }
\author{V.N. Temlyakov \thanks{ University of South Carolina. Research was supported by NSF grant DMS-1160841}} \maketitle
\begin{abstract}
{We discuss construction of coverings of the unit ball of a finite dimensional Banach space. The well known technique of comparing volumes gives upper and lower bounds on covering numbers. This technique does not provide a construction of good coverings.  Here we apply incoherent dictionaries for construction of good coverings. We use the following strategy. First, we build a good covering by balls with a radius close to one. Second, we iterate this construction to obtain a good covering for any radius. We mostly concentrate on the first step of this strategy. }
\end{abstract}

\section{Introduction}

Let  $X$ be a Banach space $\R^d$ with a norm $\|\cdot\|$ and let $B:=B_X$ denote the corresponding closed unit ball:
\begin{equation}\label{1.1}
B:=B_X:=\{x\in \R^d:\|x\|\le 1\}.
\end{equation}
The open unit ball will be denoted by $B^o:=B^o_X$:
\begin{equation}\label{1.2}
B^o:=B^o_X:=\{x\in \R^d:\|x\|< 1\}.
\end{equation}
Notation $B(x,r):=B_X(x,r)$ and $B^o(x,r):=B^o_X(x,r)$ will be used respectively for closed and open balls with the center $x$ and radius $r$. In case $r=1$ we drop it from the notation: $B^o(x):=B^o(x,1)$.
For a compact set $A$ and a positive number $\e$ we define the covering number $N_\e(A)$
 as follows
$$
N_\e(A) := N_\e(A,X) 
:=\min \{n : \exists x^1,\dots,x^n : A\subseteq \cup_{j=1}^n B_X(x^j,\e)\}.
$$
The following proposition is well known.
\begin{Proposition}\label{P1.1} For any $d$-dimensional Banach space $X$ we have
$$
\e^{-d} \le N_\e(B_X,X) \le (1+2/\e)^d.
$$
\end{Proposition}
This proposition describes the behavior of $N_\e(B_X,X)$ when $\e\to 0$. In this paper we concentrate on 
the case when $\e$ is close to $1$. In particular, we discuss the following problem: How many balls $B^o(x^j)$ are needed
for covering $B$? In other words we are interested in the number
\begin{equation}\label{1.3}
N(d,X):=\min \{n : \exists x^1,\dots,x^n : B_X \subset \cup_{j=1}^n B_X^o(x^j).
\end{equation}

We prove here that if $X$ is a uniformly smooth Banach space then $N(d,X)=d+1$. With this result in hands we discuss 
the problem: How small $\e$ can be for the relation $N_\e(B)=d+1$ to hold? The left inequality in Proposition \ref{P1.1} gives 
the lower bound for such $\e$: $\e\ge 1-\frac{\ln(d+1)}{d}$. In Section 3 we prove an upper bound: $\e \le 1- Cd^{-2}$. This upper bound follows from two different constructions given in Propositions \ref{P3.2} and \ref{P3.4}. In both constructions we use a system $\D:=\{g^j\}_{j=1}^{d+1}$ of vectors and built a covering of $B_2$ in the form $\cup_{j=1}^{d+1}B^o_2(ag^j,r)$ with an appropriate $r$. In Section 4 we apply this idea with $\D$ being an incoherent 
dictionary for covering in the Hilbert space $\ell^d_2$. We prove the following bound in Corollary \ref{C4.1}. For $r=(1-\mu^2)^{1/2}$, $\mu\in [(2n)^{-1/2},1/2]$, we have
\begin{equation}\label{1.4}
N_r(B_2) \le 2\exp(C_1d\mu^2\ln(2/\mu)).
\end{equation}

In Section 5 we use incoherent dictionaries in a smooth Banach space $X$ to build a good covering for $B_X$. Let $\rho(u)$ denote the modulus of smoothness of $X$ (see Section 3 below for definition) and $a(\mu)$ be a solution (actually, it is a unique solution) to the equation
$$
a\mu = 4\rho(2a).
$$
We prove the following bound in Corollary \ref{C5.1}. For $r=1- \frac{1}{2}\mu a(\mu)$, $\mu\le 1/2$, we have
\begin{equation}\label{1.5}
N_r(B_X) \le 2\max(C_2d,\exp(C_2d\mu^2\ln(2/\mu))).
\end{equation}
It is interesting to note (see Section 6) that in the case $X:=\ell^d_p$, $p\in[2,\infty)$, we have $1-r=\frac{1}{2}\mu a(\mu)\asymp \mu^2$ as in the case
$X=\ell^d_2$.

In Section 6 we consider several specific examples of $X$ and make a conclusion that the technique based on extremal incoherent dictionaries works well and provides either optimal or close to optimal  bounds in the sense of order of $\ln N_\e(B_X)$.  

\section{Lower bounds}

We prove the following bound in this section.

\begin{Theorem}\label{T2.1} Let  $X$ be a Banach space $\R^d$ with a norm $\|\cdot\|$. Then 
$$
N(d,X) \ge d+1.
$$
\end{Theorem}
\begin{proof} We prove that any $d$ balls $B^o(x^j)$, $j=1,\dots,d$ do not cover $B$. Indeed, for a given 
set $\{B^o(x^j)\}_{j=1}^d$ consider the linear manifold $M$ passing through $x^1,\dots,x^d$:
$$
M:=\{x : x=x^1+t_1(x^2-x^1)+\dots+t_{d-1}(x^d-x^1),\quad t_j\in \R\}.
$$
It is clear that $M$ is a $(d-1)$-dimensional linear manifold. We use  Lemma \ref{2.1} below which guarantees that there is $z\in B$, $\|z\|=1$ such that for any $x\in M$ we have $\|z-x\|\ge 1$.
Then $z\in B$ is not covered by the $\cup_{j=1}^d B_X^o(x^j)$. 
\end{proof}
\begin{Lemma}\label{L2.1} Let  $X$ be a Banach space $\R^d$ with a norm $\|\cdot\|$. Then for any $(d-1)$-dimensional 
manifold $M$ we have
\begin{equation}\label{2.1} 
d(B_X,M):=\sup_{y\in B_X}\inf_{x\in M} \|y-x\|\ge 1. 
\end{equation}
\end{Lemma}
\begin{proof} Without loss of generality we can assume that $M$ is a subspace. Indeed, by symmetry of $B_X$ we have
that $d(B_X,M)=d(B_X,M^-)$ where
$$
M^-:=\{x : -x\in M\}.
$$
Let
$$
M=\{x : x=x^0+t_1u^1+\dots+t_{d-1}u^{d-1},\quad t_j\in \R\}.
$$
Define a subspace 
$$
M^0:=\{x : x = t_1u^1+\dots+t_{d-1}u^{d-1},\quad t_j\in \R\}.
$$
Then $d(B_X,M^0)\le d(B_X,M)$. Indeed, for any $y\in B_X$ there are $x^+\in M$ and $x^-\in M^-$ such that
$$
\|y-x^+\|\le d(B_X,M),\qquad \|y-x^-\|\le d(B_X,M^-) = d(B_X,M).
$$
Set $x^0:= (x^++x^-)/2\in M^0$. Then
$$
\|y-x^0\|\le \|y-x^+\|/2+\|y-x^-\|/2\le d(B_X,M).
$$
So, we assume that $M$ is a subspace. A standard proof of statements like Lemma \ref{L2.1} is based on the 
antipodality theorem of Borsuk (see, for instance, \cite{LGM}, p. 405). We give a proof that is based on ideas from functional analysis. Let $w$ be a functional such that $\|w\|_{X^*}=1$ and $w(x)=0$ for $x\in M$.  Consider a norming functional $F_w$ for $w$. Our space $X$ is a reflexive Banach space. So $F_w\in B_X$. For any $x\in M$ we have
$$
\|F_w-x\| \ge |w(F_w-x)| = 1.
$$
This completes the proof of Lemma \ref{2.1}
\end{proof}

\section{Upper bounds}

We begin with the case when the norm $\|\cdot\|= \|\cdot\|_2$ is the Euclidean norm. Let $\{e^j\}_{j=1}^d$ denote 
the standard basis: $e^j_i=0$ if $i\neq j$ and $e^j_j=1$. 
\begin{Proposition}\label{P3.1} Define $x^j:=\frac{1}{2d}e^j$, $j=1,\dots,d$ and $x^{d+1}:= -\frac{1}{2d}\sum_{j=1}^d e^j$.
Then
$$
B_2\subset \cup_{j=1}^{d+1} B^o_2(x^j).
$$
\end{Proposition}
\begin{proof} We begin with describing a set that is not covered by $B^o_2(x^k)$, $k\in[1,d]$. Take any point $y\in B_2$.
Then $\sum_{j=1}^d y_j^2 \le 1$. Setting $a:=\frac{1}{2d}$ we obtain
$$
\|y-x^k\|^2 = \sum_{j\neq k}y_j^2 +(y_k-a)^2.
$$
If $(y_k-a)^2 <y_k^2$ then $y\in B^o(x^k)$. Thus those $y_k$ which are not covered by $B^o(x^k)$ satisfy
the inequality $(y_k-a)^2 \ge y_k^2$ which implies $y_k\le a/2$. Therefore,
$$
B_2\setminus \cup_{k=1}^{d} B^o_2(x^k) \subset C:=\{y : y\in B_2, y_k\le a/2\}.
$$
We now prove that $C\subset B^o_2(x^{d+1})$. Indeed, for any $y\in C$ we have
$$
b:=\sum_{k=1}^d(y_k+a)^2 = \sum_{k=1}^d y_k^2 +2a\sum_{k=1}^d y_k +da^2.
$$
The inequality $y_k\le a/2$ implies $y_k \le -|y_k| + a$ and
$$
b\le \sum_{k=1}^d y_k^2 - 2a\sum_{k=1}^d |y_k| +3da^2.
$$
Using
$$
\sum_{k=1}^d |y_k| \ge \sum_{k=1}^d y_k^2
$$
we obtain
$$
b\le (1-2a)\sum_{k=1}^d y_k^2 +3da^2 \le 1-2a +3da^2\le 1-\frac{1}{4d}.
$$
\end{proof}

\begin{Proposition}\label{P3.2} Define $a:=\frac{2}{5d+1}$, $x^j:=ae^j$, $j=1,\dots,d$ and $x^{d+1}:= -a\sum_{j=1}^d e^j$.
Then
$$
B_2\subset \cup_{j=1}^{d+1} B^o_2(x^j,r)\quad \text{with}\quad r>(1-a^2)^{1/2}.
$$
\end{Proposition}
\begin{proof} The proof repeats the proof of Proposition \ref{P3.1}. We only point out the places where we make changes.
First, we note that if $y_k>a$ then $y_k^2 - (y_k-a)^2 > a^2$. Therefore, in this case $y\in B^o_2(x^k,r)$. We have
$$
B_2\setminus \cup_{k=1}^{d} B^o_2(x^k,r) \subset C':=\{y : y\in B_2, y_k\le a\}.
$$
We now prove that $C'\subset B^o_2(x^{d+1},r)$. Similar to the above argument we get
$$
b\le (1-2a)\sum_{k=1}^d y_k^2 +5da^2 \le 1-2a +5da^2= 1-a^2 <r^2.
$$
\end{proof}

For a Banach space $X$ we define the modulus of smoothness
$$
\rho(u) := \sup_{\|x\|=\|y\|=1}(\frac{1}{2}(\|x+uy\|+\|x-uy\|)-1).
$$
The uniformly smooth Banach space is the one with the property
$$
\lim_{u\to 0}\rho(u)/u =0.
$$

\begin{Proposition}\label{P3.3} Let $X$ be a uniformly smooth Banach space $\R^d$ with norm $\|\cdot\|$.  Define $x^j:=ae^j$, $j=1,\dots,d$ and $x^{d+1}:= -a\sum_{j=1}^d e^j$.
Then there exists an $a>0$ such that
\begin{equation}\label{3.1}
B\subset \cup_{j=1}^{d+1} B^o(x^j).
\end{equation}
\end{Proposition}
\begin{proof} Embedding (\ref{3.1}) is equivalent to the claim that for each $y\in B$ at least one of the following $d+1$ inequalities 
is satisfied
\begin{equation}\label{3.2}
\|y-ae^j\|<1,\quad j\in [1,d];
\end{equation}
\begin{equation}\label{3.3}
\|y+a\sum_{j=1}^de^j\| < 1.
\end{equation}
In the proof that follows parameter $a$ is small. We assume that $a<1/2$. Then for $y$ such that $\|y\|\le 1/2$ all inequalities
(\ref{3.2}) are satisfied. Therefore, in further argument it is sufficient to consider $y$ such that $1/2<\|y\|\le 1$. 

For $x\neq 0$ let $F_x$ be a norming functional for $x$: $\|F_x\|_{X^*}=1$ and $F_x(x)=\|x\|$. Existence of such a functional follows from the Hahn-Banach theorem.  
We note that from the definition of modulus of smoothness we get the following inequality (see, for instance, \cite{Tbook}, p.336).
 
\begin{Lemma}\label{L3.1} Let $x\neq0$. Then
$$
0\le \|x+uy\|-\|x\|-uF_x(y)\le 2\|x\|\rho(u\|y\|/\|x\|)  
$$
where $F_x$ is a norming functional of $x$.
\end{Lemma}
This lemma implies the following inequalities
$$
\|y-ae^j\| \le \|y\| -a F_y(e^j)+2\|y\|\rho(a/\|y\|)
$$
\begin{equation}\label{3.4}
 \le  \|y\| -a F_y(e^j)+2\|y\|\rho(2a),\quad j\in [1,d];
\end{equation}
\begin{equation}\label{3.5}
\|y+a\sum_{j=1}^de^j\| \le \|y\| + aF_y(\sum_{j=1}^d e^j) +2\rho(2da).
\end{equation}
Here, $F_y$ is the norming functional of $y$. 

First, we note that for some $k$ the $|F_y(e^k)|$ is large enough. Indeed, let $y=\sum_{j=1}^d y_je^j$. Then
$$
|y_j| \le C_1(d) \|y\|,\quad j=1,\dots,d.
$$
We have
$$
\|y\| = F_y(y) = \sum_{j=1}^d y_jF_y(e^j) \le C_1(d)\|y\|\sum_{j=1}^d |F_y(e^j)|,
$$
which implies that for some $k\in[1,d]$ 
\begin{equation}\label{3.6}
|F_y(e^k)| \ge (dC_1(d))^{-1} =: c_1.
\end{equation} 
Set $b:=c_1/2$ and consider three cases:
\begin{equation}\label{3.7}
F_y(\sum_{j=1}^de^j) \le -b,
\end{equation}
\begin{equation}\label{3.8}
F_y(\sum_{j=1}^de^j) \ge b,
\end{equation}
\begin{equation}\label{3.9}
|F_y(\sum_{j=1}^de^j)| < b.
\end{equation}
In the case (\ref{3.7}) inequality (\ref{3.5}) implies (\ref{3.3}) if $a:=a(b,\rho,d)$ is sufficiently small (remind that 
uniform smoothness assumption implies $\rho(u)/u \to 0$ as $u\to 0$). In the case (\ref{3.8}) we have for some $k\in [1,d]$
that $F_y(e^k) \ge b/d$ and this is sufficient to derive (\ref{3.2}) with $j=k$ from (\ref{3.4}) and small $a$. 

Consider the case (\ref{3.9}). Inequality (\ref{3.6}) guarantees that either $F_y(e^k)\ge c_1$ or $-F_y(e^k)\ge c_1$. In case 
$F_y(e^k)\ge c_1$ we complete the proof as in case (\ref{3.8}). In case $-F_y(e^k)\ge c_1$ our assumption (\ref{3.9}) implies that
$$
\sum_{j=1}^dF_y(e^j)> -b
$$
and
$$
\sum_{j\neq k}^dF_y(e^j) > -b -F_y(e^k) \ge -b+c_1 = c_1/2.
$$
Therefore, for some $m$
$$
F_y(e^m)\ge \frac{c_1}{2(d-1)}
$$
and we complete the proof as in case (\ref{3.8}).
\end{proof}

We now discuss another way of constructing a $(d+1)$-covering of the Euclidean ball. It is based on the tight frames 
construction. We begin with a conditional statement.
\begin{Proposition}\label{P3.4} Let $\Phi:=\{\ff^j\}_{j=1}^{d+1}$ be a system of normalized vectors, $\|\ff^j\|_2=1$, $j=1,\dots,d+1$, satisfying the condition
$$
\<\ff^i,\ff^j\> = -\frac{1}{d},\quad 1\le i\neq j\le d+1.
$$
Then, there exists an $a>0$ such that
$$
B_2 \subset \cup_{j=1}^{d+1} B^o_2(a\ff^j).
$$
\end{Proposition}
\begin{proof} In our proof $a$ is a small number. Let $a<1/2$. Then for any $x$, $\|x\|_2\le 1/2$, and any $k\in [1,d+1]$ we have
$$
\|x-a\ff_k\|_2 <1.
$$
Thus, it is sufficient to consider $x$ such that $1/2\le \|x\|_2\le 1$. For each $k$ we have
\begin{equation}\label{3.10}
\|x-a\ff^k\|^2_2 = \|x\|_2^2 +a^2 - 2a\<x,\ff^k\>.
\end{equation}
We now need to estimate $\<x,\ff^k\>$ from below. It is easy to check that our assumptions on $\Phi$ imply 
the relations
\begin{equation}\label{3.11}
x=\frac{d}{d+1}\sum_{i=1}^{d+1}\<x,\ff^i\>\ff^i,
\end{equation}
\begin{equation}\label{3.12}
\sum_{i=1}^{d+1}\ff^i=0,
\end{equation}
\begin{equation}\label{3.13}
\|x\|_2^2=\frac{d}{d+1}\sum_{i=1}^{d+1}\<x,\ff^i\>^2.
\end{equation}
We now need a simple technical lemma.
\begin{Lemma}\label{L3.2} If $y\in \R^N$ is such that $\sum_{i=1}^N y_i=0$ then there exists $k$ satisfying
$$
y_k\ge \frac{\|y\|_2}{2(N-1)}.
$$
\end{Lemma}
\begin{proof} The proof goes by contradiction. Suppose $y_j< \frac{\|y\|_2}{2(N-1)}$ for all $j$. Denote
$$
E^+:=\{j : y_j>0\},\qquad E^-:=\{j : y_j<0\}.
$$
Then our assumption implies (note that $|E^+|\le N-1$)
$$
\sum_{j\in E^+} y_j < \|y\|_2/2,
$$
and, therefore,
$$
\|y\|_1 = \sum_{j=1}^N |y_j| = 2\sum_{j\in E^+} y_j < \|y\|_2.
$$
It is a contradiction.  
\end{proof}

We apply Lemma \ref{L3.2} with $N:=d+1$, $y_j:=\<x,\ff^j\>$. Then the condition $\sum_{i=1}^N y_i=0$ follows from
(\ref{3.12}). Thus, by (\ref{3.13}), taking into account that $\|x\|_2\ge 1/2$, we derive from Lemma \ref{L3.2} that there exists $k$ such that
$$
\<x,\ff^k\> \ge (2d)^{-1}(\sum_{i=1}^{d+1}\<x,\ff^i\>^2)^{1/2} \ge \frac{1}{4d}.
$$
By (\ref{3.10}) we obtain for this $k$ 
$$
\|x-a\ff^k\|_2^2 \le 1+a^2-\frac{a}{4d}.
$$
Specifying $a=\frac{1}{8d}$ we get
$$
\|x-a\ff^k\|_2^2 \le 1-\frac{1}{64d^2}.
$$
\end{proof}

We now discuss a question of existence and construction of systems $\Phi$
from Proposition \ref{P3.4}. We only give one example of such construction which 
is based on the Hadamard matrices.
Hadamard matrices are very useful in both theoretical research and
engineering applications. In particular, Hadamard matrices are very
popular in error-correction coding theory.  
A {\it Hadamard matrix} of order $n$ is an $n\times n$ matrix $H_n$
with all entries $1$ or $-1$, and
\begin{equation*} 
H_n^{T}H_n \;=\; nI_n
\end{equation*}
where $I_n$ is the identity matrix.
Obviously, any two columns or any two rows of a Hadamard matrix
$H_n$ are mutually orthogonal. This orthogonality is kept
  if we permute some rows or columns, or multiply some
rows or columns by -1. Therefore, given any Hadamard matrix, we can
always make a new Hadamard matrix which has all 1's in the first row
   by
multiplying some columns   by -1.
  Hadamard matrices only exist for special orders $n$.  The following lemma  and remark are from \cite{LW01}.

\begin{Lemma}\label{L3.3}  
If $H_n$ is a Hadamard matrix of order $n$, then $n\;=\;1$,
$n\;=\;2$, or $n\;\equiv\;0\;(mod\; 4)$.
\end{Lemma}

\begin{Remark}\label{R3.1}
One of the famous conjectures in the area of combinatorial designs
states that a Hadamard matrix of order $n$ exists for every $n\equiv
0 \;(mod\; 4)$. But we are still very far from a proof of this
conjecture. The smallest $n$ for which a Hadamard matrix could exist
but no example is known presently 428.
\end{Remark}

There exists a variety of methods to construct Hadamard matrices. We can construct Hadamard matrices from so-called
{\it conference matrices} (see \cite{LW01}). We will not discuss this way. 
For illustration purposes we provide a very simple construction of Hadamard matrices of order $2^k$. 
The following lemma provides a recursive method to
build Hadamard matrices of order $2^k$, where $k=0,1,2,...$.

\begin{Lemma} \label{L3.4}
For $k=0, 1, 2, ...$, the matrices generated by
\begin{align*} 
&H_1 \;=\; 
\begin{bmatrix}
1
\end{bmatrix},\\
&H_2 \;=\;
\begin{bmatrix}
1 & 1\\
1 & -1
\end{bmatrix},\\
&H_{2^{k+1}} \;=\;
\begin{bmatrix}
H_{2^k} & H_{2^k} \\
H_{2^k} & -H_{2^k}
\end{bmatrix},  
\end{align*}
are Hadamard matrices.
\end{Lemma}
\begin{proof} Clearly, $H_1$ and $H_2$ are Hadamard matrices of order 1 and 2
respectively.
Assume $H_{2^k}$ is a Hadamard matrix of order $2^k$, then
\begin{equation*}
H_{2^k}^{T}H_{2^k} \;=\; 2^kI_{2^k}.
\end{equation*}
We need to show that
\begin{equation*} 
H_{2^{k+1}}^{T}H_{2^{k+1}} \;=\; 2^{k+1}I_{2^{k+1}}.
\end{equation*}
Indeed,
\begin{align*}
H_{2^{k+1}}^{T}H_{2^{k+1}} \;&=\;
\begin{bmatrix}
H_{2^k} & H_{2^k} \\
H_{2^k} & -H_{2^k}
\end{bmatrix}^T
\begin{bmatrix}
H_{2^k} & H_{2^k} \\
H_k & -H_k
\end{bmatrix}\\
&=\;
\begin{bmatrix}
H_{2^k}^T & H_{2^k}^T \\
H_{2^k}^T & -H_{2^k}^T
\end{bmatrix}
\begin{bmatrix}
H_{2^k} & H_{2^k} \\
H_{2^k} & -H_{2^k}
\end{bmatrix} \\
&=\;
\begin{bmatrix}
2H_{2^k}^{T}H_{2^k} & 0 \\
0 & 2H_{2^k}^{T}H_{2^k}
\end{bmatrix}\\
&=\;
\begin{bmatrix}
2^{k+1}I_{2^k} & 0 \\
0 & 2^{k+1}I_{2^k}
\end{bmatrix}\\
&=\; 2^{k+1}I_{2^{k+1}}.
\end{align*}
 
\end{proof}

We can build higher order Hadamard matrices from the
Kronecker product of lower order Hadamard matrices.
Let matrix $A\in R^{n\times m}$ with entries $a_{ij}$ and $B\in
R^{l\times k}$. Then the \emph{Kronecker product} $A \otimes B$ of
$A$ and $B$ is a $nl\times mk$ matrix,
\begin{equation*} 
A \otimes B \;=\;
\begin{bmatrix}
a_{11}B & a_{12}B & \cdots  & a_{1m}B\\
a_{21}B & a_{22}B & \cdots  & a_{2m}B\\
\vdots & \vdots & \ddots & \vdots\\
a_{n1}B & a_{n2}B & \cdots  & a_{nm}B\\
\end{bmatrix}.
\end{equation*}
The following simple lemma is known.
\begin{Lemma} \label{L3.5}
If $H_m$ and $H_n$ are Hadamard matrices of order $m$ and $n$
respectively, then $H_m \otimes H_n$ is a Hadamard matrix of order
$mn$.
\end{Lemma}
 This lemma provides a good way to
build higher order Hadamard matrices from known lower order ones. We
can see that Lemma \ref{L3.4} is  a corollary of
Lemma \ref{L3.5}, where the recursion is
$H_{2^{k+1}}\;=\;H_2\otimes H_{2^k}$.

The Hadamard matrices were used in \cite{XL} for construction systems from Proposition \ref{P3.4}.
Such systems are called absolutely equiangular tight frames in \cite{XL}.  
\begin{Theorem}\label{T3.1}
Let $H_m$ be a Hadamard matrix with all $1's$ in the first row and
$m=n+1$. Then, the columns of the matrix $\Phi$ generated by deleting the first row
of $H_m$ and dividing by $\sqrt{n}$ form an absolutely equiangular
tight frame.
\end{Theorem}
\begin{proof}
   All columns of $H_m$ are mutually orthogonal. In other words,
for any $1\leq i\neq j \leq m$, the two columns $h_i$ and $h_j$ of
$H_m$ satisfy $\<h_i, h_j\> \;=\; 0$.

Since the first elements of $h_i$ and $h_j$ are both 1, the
corresponding columns $\varphi_i$ and $\varphi_j$ of $\Phi$ satisfy
\begin{equation*}
\<\varphi_i, \varphi_j\> \;=\; \frac{1}{n}(\<h_i, h_j\>-1)
\;=\;\frac{1}{n}(0-1) \;=\; -\frac{1}{n},
\end{equation*}
for all $1\leq i\neq j \leq m$.
\end{proof}

\section{Covering using incoherent dictionaries}

Proposition \ref{P3.4} demonstrates how special dictionaries can be used for building coverings. In this section we discuss an application of incoherent dictionaries in Euclidean space. Let $\D=\{g^k\}_{k=1}^N$ be a normalized ($\|g^k\|=1$, $k=1,\dots,N$) system of vectors in $\R^d$  equipped with the Euclidean norm.
We define the coherence parameter of the dictionary $\D$ as follows
$$
M(\D) := \sup_{k\neq l} |\<g^k,g^l\>|.
$$
    
In this section we discuss the following   characteristics
$$
N(d,\mu) := \sup\{N:\exists \D \quad\text{such that} \quad\# \D \ge N, M(\D)\le\mu\}.
$$

The problem of studying $N(d,\mu)$ is equivalent to a fundamental problem of information theory. It is a problem on optimal spherical codes. A spherical code ${\mathcal S}(d,N,\mu)$ is a set of $N$ points (code words) on the $d$-dimensional unit sphere, such that the absolute values of inner products between any two distinct code words is not greater than $\mu$. The problem is to find the largest $N^*$ such that the spherical code  ${\mathcal S}(d,N^*,\mu)$ exists. It is clear that $N^*=N(d,\mu)$. Denote by $\D(\mu)$ a dictionary such that $M(\D(\mu))\le \mu$ and $|\D(\mu)|=N(d,\mu)$. We call such $\D(\mu)$ an {\it extremal dictionary} for a given $\mu$. 
\begin{Theorem}\label{T4.1} Let $\D(\mu):=\{g^k\}_{k=1}^{N(d,\mu)}$ be an extremal dictionary for a given $\mu\le (1/2)^{1/2}$. Then 
$$
B_2\subset (\cup_{j=1}^{N(d,\mu)}B^o_2(\mu g^j,r))\cup (\cup_{j=1}^{N(d,\mu)}B^o_2(-\mu g^j,r)),\quad r^2=1-\mu^2.
$$
Thus, $N_r(B_2)\le 2N(d,\mu)$.
\end{Theorem}
\begin{proof} Our assumption that $\D(\mu)$ is an extremal dictionary for $\mu$ implies that for any $x\in B_2$ there is 
$g^k\in \D(\mu)$ such that $|\<x/\|x\|_2,g^k\>| > \mu$. Suppose, $\<x/\|x\|_2,g^k\> > \mu$. The other case $\<x/\|x\|_2,-g^k\> > \mu$ is treated exactly the same way. Then
$$
\|x-\mu g^k\|_2^2 = \|x\|_2^2 + \mu^2 -2\mu\<x,g^k\> < \|x\|^2_2 +\mu^2 -2\mu^2\|x\|_2 \le 1-\mu^2.
$$
\end{proof}

The problem of estimating $N(d,\mu)$ is well studied (see, for instance, \cite{Tbook}, section 5.7, p. 314). It is known (see \cite{Tbook}, p. 315) that for a system $\D$ with $\#\D\ge 2n$ we have $M(\D)\ge (2n)^{-1/2}$. Thus, a natural range for $\mu$ is $[(2n)^{-1/2},1]$. In particular,
the following bound is known (see \cite{Tbook}, p. 315)
\begin{equation}\label{4.1}
N(d,\mu)\le \exp(C_1d\mu^2\ln(2/\mu)), \quad \mu \in [(2n)^{-1/2},1/2]. 
\end{equation}
As a corollary of (\ref{4.1}) and Theorem \ref{T4.1} we obtain the following statement.
\begin{Corollary}\label{C4.1} For $r=(1-\mu^2)^{1/2}$, $\mu\in [(2n)^{-1/2},1/2]$, we have
$$
N_r(B_2) \le 2\exp(C_1d\mu^2\ln(2/\mu)).
$$
\end{Corollary}

\section{Covering in Banach spaces using incoherent dictionaries}

 We use here a generalization of the concept of $M$-coherent dictionary to the case of Banach spaces. This generalization was published in \cite{T23} (see also \cite{Tbook}, p. 381).

Let $\D$ be a dictionary in a Banach space $X$. We define the coherence parameter of this dictionary in the following way
$$
M(\D):=M(\D,X):= \sup_{g\neq h;g,h\in\D}\sup_{F_g}|F_g(h)|,
$$
where $F_g$ is a norming functional for $g$. We note that, in general, a norming functional $F_g$ is not unique. This is why we take $\sup_{F_g}$ over all norming functionals of $g$ in the definition of $M(\D)$. We do not need $\sup_{F_g}$ in the definition of $M(\D)$ if for each 
$g\in\D$ there is a unique norming functional $F_g\in X^*$. Then we define $\D^*:=\{F_g,g\in\D\}$ and call $\D^*$ a {\it dual dictionary} to a dictionary $\D$. 
 It is known that the uniqueness of the norming functional $F_g$ is equivalent to the property that $g$ is a point of Gateaux smoothness:
 $$
 \lim_{u\to 0}(\|g+uy\|+\|g-uy\|-2\|g\|)/u =0
 $$
 for any $y\in X$. In particular, if $X$ is uniformly smooth then $F_f$ is unique for any $f\neq 0$. 

Let $\D:=\{g^j\}_{j=1}^N$ be a normalized system of vectors in $X$, which is $\R^d$ equipped with a norm $\|\cdot\|$, 
$g^j=(g^j_1,\dots,g^j_d)^T$. Denote by
$$
\Phi:=[g^1,\dots,g^N]
$$
a $d\times N$ matrix formed by column vectors $\{g^j\}$. Suppose for simplicity that for each $g^j$ there is a unique norming functional $F_{g^j}\in X^*$. Each functional $F_{g^j}\in X^*$ can be associated with a vector $w^j\in \R^d$ in such a way that $w_i^j=F_{g^j}(e^i)$, $i=1,\dots,d$. Then 
$$
F_{g^j}(g^k) = \sum_{i=1}^d g^k_iF_{g^j}(e^i)=\sum_{i=1}^d g^k_iw^j_i = \<w^j,g^k\>.
$$
Consider the matrix
$$
W:=[w^1,\dots,w^N] 
$$
which is a $d\times N$ matrix formed by column vectors $\{w^j\}$. Consider the transposed matrix $W^T$ that is formed by the row vectors $(w^j_1,\dots,w^j_d)$, $j=1,\dots,N$, or by the column vectors $h_i:=(w^1_i,\dots,w^N_i)^T$, $i=1,\dots, d$.
Define the {\it coherence matrix} of a dictionary $\D$ 
as follows
$$
C(\D):=W^T\Phi.
$$
 Then the coherence matrix $C(\D)$ of the system $\D=\{g^j\}_{j=1}^N$ satisfies the following inequality for the rank: $\rank C(\D) \le d$.
 Indeed, the columns of $C(\D)$ are linear combinations of $d$ columns $h_i$, $i=1,\dots, d$. 
 It is clear that the coherence matrix $C(\D)=||c_{i,j}||_{i=1,j=1}^{N}$, $c_{i,j}=F_{g^i}(g^j)$, has $1$ on the diagonal and for all off-diagonal elements we have $|c_{i,j}|\le M(\D)$. 
 
 In this section we discuss the following   characteristics
$$
N(d,\mu,X) := \sup\{N:\exists \D \quad\text{such that} \quad\# \D \ge N, M(\D,X)\le\mu\}.
$$

We now use a fundamental result of Alon (see, for instance, \cite{Tbook}, p.317) to derive an upper bound for $N(n,\mu,X)$ from the property $\rank C(\D)\le d$. 
 
 \begin{Theorem}\label{T5.1}   Let $A:=\|a_{i,j}\|_{i,j=1}^N$ be a square matrix of the form $a_{i,i}=1$, $i=1,\dots,N$; $|a_{i,j}|\le\e<1/2$, $i\neq j$. Then
 \begin{equation}\label{5.1}
 \min(N,(\ln N)(\e^2\ln(2/\e))^{-1}) \le C_2\rank A  
 \end{equation}
 with an absolute constant $C_2$.
 \end{Theorem}
 
 We apply this theorem with $A=C(\D)$ and $\e=\mu$. For \newline $N\le (\ln N)(\mu^2\ln(2/\mu))^{-1}$ (\ref{5.1}) implies that
 $$
 N\le C_2d.
 $$
 For $N\ge (\ln N)(\mu^2\ln(2/\mu))^{-1}$ (\ref{5.1}) implies that
 $$
 (\ln N)(\mu^2\ln(2/\mu))^{-1} \le C_2d
 $$
 and
 \begin{equation}\label{5.2}
 N\le \exp(C_2d\mu^2\ln(2/\mu)).  
 \end{equation}
 Thus,
 $$
 N\le \max(C_2d,\exp(C_2d\mu^2\ln(2/\mu))).
 $$
 We formulate the above result as a theorem.
 \begin{Theorem}\label{T5.2} For a Banach space $X$ which is $\R^d$ equipped with a norm $\|\cdot\|$ we have
 \begin{equation}\label{5.3}
 N(d,\mu,X)\le \max(C_2d,\exp(C_2d\mu^2\ln(2/\mu))).
 \end{equation}
 \end{Theorem}
 
In particular, in the case $\mu = C_3d^{-1/2}$,  inequality (\ref{5.3}) gives the polynomial bound $N(d,\mu,X)\le d^{C_4}$. 

Let $X$ be a uniformly smooth Banach space with modulus of smoothness $\rho(u)$. Denote by 
$a(\mu)$ a solution (actually, it is a unique solution) to the equation
$$
a\mu = 4\rho(2a)
$$
if it exists and set $a(\mu):=1$ otherwise. Then we always have $4\rho(2a(\mu))\le a(\mu)\mu$. 
Denote by $\D(\mu,X)$ a dictionary such that $M(\D(\mu),X)\le \mu$ and $|\D(\mu,X)|=N(d,\mu,X)$. We call such $\D(\mu,X)$ an {\it extremal dictionary} for a given $\mu$ in the space $X$. 
\begin{Theorem}\label{T5.3} Let $\D(\mu,X):=\{g^k\}_{k=1}^{N(d,\mu,X)}$ be an extremal dictionary for a given $\mu$ in the space $X$. Then 
$$
B_X\subset (\cup_{j=1}^{N(d,\mu,X)}B^o_X(a(\mu) g^j,r))\cup (\cup_{j=1}^{N(d,\mu,X)}B^o_X(-a(\mu) g^j,r)),\quad r=1-\frac{1}{2}\mu a(\mu).
$$
Thus, $N_r(B)\le 2N(d,\mu,X)$ with $r=1-\frac{1}{2}\mu a(\mu)$.
\end{Theorem}
\begin{proof} Our assumption that $\D(\mu,X)$ is an extremal dictionary for $\mu$ in the space $X$ implies that for any $x\in B_X$ there is 
$g^k\in \D(\mu,X)$ such that $|F_x(g^k)| > \mu$. Suppose, $F_x(g^k) > \mu$. The other case $F_x(-g^k) > \mu$ is treated exactly the same way. Without loss of generality we assume that $\|x\|\ge 1/2$. Then by Lemma \ref{L3.1} we get
$$
\|x-a(\mu)g^k\| \le \|x\| -a(\mu) F_x(g^k)+2\|x\|\rho(2a(\mu))
$$
$$
 < 1-\mu a(\mu) +2\rho(2a(\mu)) \le 1- \frac{1}{2}\mu a(\mu).
$$
\end{proof}

As a corollary of Theorem \ref{T5.2} and Theorem \ref{T5.3} we obtain the following statement.
\begin{Corollary}\label{C5.1} For $r=1- \frac{1}{2}\mu a(\mu)$, $\mu\le 1/2$, we have
$$
N_r(B_X) \le 2\max(C_2d,\exp(C_2d\mu^2\ln(2/\mu))).
$$
\end{Corollary}

\begin{Remark}\label{R5.1} Let $\oo(u)$ be a continuous majorant of $\rho(u)$, $u\in[0,\infty)$, such that $\oo(u)/u$ monotone decreases  to $0$ as $u\to 0$. Define $a(\oo,\mu)$ as a solution to the equation
\begin{equation}\label{5.4}
a\mu = 4\oo(2a)
\end{equation}
if it exists and set $a(\oo,\mu)=1$ otherwise. 

Theorem \ref{T5.3} and Corollary \ref{C5.1} hold with $a(\mu)$ replaced by $a(\oo,\mu)$.
\end{Remark}

\section{Some examples}

In this section we discuss the above results demonstrating their power on some specific examples. 

{\bf Example 1.} Assume that $X$, being a uniformly smooth Banach space $\R^d$ with norm $\|\cdot\|$, has modulus of smoothness of power type:
$\rho(u)\le \gamma u^q$, $q\in (1,2]$. Setting $\oo(u):=\gamma u^q$ we find
$$
a(\oo,\mu) = \left(\frac{\mu}{\gamma 2^{q+2}}\right)^{\frac{1}{q-1}}.
$$
By Remark \ref{R5.1} and Corollary \ref{C5.1} we get for $r=1- \frac{1}{2}\mu a(\oo,\mu)$, $\mu\le 1/2$, 
$$
N_r(B_X) \le 2\max(C_2d,\exp(C_2d\mu^2\ln(2/\mu))).
$$
In other words, denoting $\delta:=1-r$ and $q':=\frac{q}{q-1}$ we obtain
\begin{equation}\label{6.1}
N_{1-\delta}(B_X) \le 2\max(C_2d,\exp(C_2(q,\gamma)d\delta^{2/q'}\ln(2/\delta))).
\end{equation}
In particular, if $q=2$ and $\delta=\frac{A_1}{d}$ we get a polynomial bound 
$$
N_{1-\frac{A_1}{d}}(B_X) \le A_2d^{A_3}.
$$
In case $q\in (1,2]$ we get a polynomial bound for $N_{1-\delta}(B_X)$ for
$\delta \asymp d^{-q'/2}$. 

{\bf Example 2.} Let $X:=\ell^d_p$, $p\in (1,\infty)$. Then it is known that 
\begin{equation}\label{6.2}
\rho(u) \le u^p/p \quad \text{if} \quad 1\le p\le 2,
\end{equation}
\begin{equation}\label{6.3}
\rho(u) \le \frac{p-1}{2}u^2 \quad \text{if} \quad 2\le p<\infty.
\end{equation}
We begin with the case $2\le p<\infty$. Specify $\oo(u):=\frac{p}{2}u^2$.
Then
$$
a(\oo,\mu) = \frac{\mu}{8p},
$$
$$
\delta := \frac{1}{2}\mu a(\oo,\mu) = \frac{\mu^2}{16p},
$$
$$
\mu =  4p^{1/2}\delta^{1/2}.
$$
Thus, by Remark \ref{R5.1} and Corollary \ref{C5.1} we get for $r=1- \frac{1}{2}\mu a(\oo,\mu)$, $\mu\le 1/2$,
\begin{equation}\label{6.4}
N_{1-\delta}(B_p^d) \le 2\max(C_2d,\exp(8C_2dp\delta \ln\frac{1}{4p\delta})).
\end{equation}
This implies that we obtain a polynomial bound for $N_{1-\delta}(B^d_p)$ in case $2\le p<\infty$ for $\delta \asymp \frac{1}{pd}$. 

In the case $p\in (1,2)$ we set $\oo(u) :=\frac{u^p}{p}$ and get
$$
a(\oo,\mu) = \left(\frac{p\mu}{2^{p+2}}\right)^{\frac{1}{p-1}},
$$
$$
\delta := \frac{1}{2}\mu a(\oo,\mu) = C(p) \mu^{p'},\qquad \mu \asymp \delta^{\frac{1}{p'}}.
$$
As above by Remark \ref{R5.1} and Corollary \ref{C5.1} we get for $p\in(1,2)$
\begin{equation}\label{6.5}
N_{1-\delta}(B_p^d) \le 2\max(C_2d,\exp(C_2(p)d\delta^{\frac{2}{p'}}\ln\frac{2}{\delta})).
\end{equation}
We obtain a polynomial bound for $N_{1-\delta}(B^d_p)$ in case $1< p<2$ for $\delta \asymp (\frac{1}{d})^{p'/2}$.

{\bf Example 3.} Let $X$ be a $d$-dimensional subspace of $L_p$, $1<p<\infty$. Similar to Example 2 we have
\begin{equation}\label{6.6}
\rho(u) \le u^p/p \quad \text{if} \quad 1\le p\le 2,
\end{equation}
\begin{equation}\label{6.7}
\rho(u) \le \frac{p-1}{2}u^2 \quad \text{if} \quad 2\le p<\infty.
\end{equation}
Therefore, relations (\ref{6.4}) and (\ref{6.5}) hold in this case too.

{\bf Example 4.} Let $X:=\ell^d_\infty$. Proposition \ref{P1.1} guarantees that for any $r\in [1/2,1)$ we have an exponential bound
$$
N_r(B_X) \le C^d
$$
for all $d$-dimensional spaces $X$ independently of their smoothness. 
It is easy to see that
$$
N(d,\ell^d_\infty) =2^d
$$
and, therefore, for all $r\in[1/2,1)$ we have
$$
N_r(B^d_\infty) \ge 2^d.
$$
This example shows that smoothness assumptions are important for breaking the exponential behavior of $N_r(B_X)$. 

The left inequality in Proposition \ref{P1.1} implies that for any $d$-dimensional Banach spaces $X$ the covering numbers $N_{1-\delta}(B_X)$ may have polynomial growth in $d$ only if $\delta \ll \frac{\ln d}{d}$. Examples 1--3 show that our technique based on extremal $\mu$-coherent dictionaries allows us to build polynomial in $d$ coverings of $B_X$ with $r=1-\delta$, $\delta \gg \frac{1}{d}$, for smooth $X$. 

{\bf Example 5.} Let $X:=\ell^d_2$. Take a dictionary $\D:=\{\pm e^j\}_{j=1}^d$.
Set $a:=\frac{1}{4d^{1/2}}$. Consider the covering
$$
\left(\cup_{j=1}^d B_2(ae^j,r)\right)\cup\left(\cup_{j=1}^d B_2(-ae^j,r)\right).
$$
We prove that there exists $c>0$ such that the above union with $r\ge 1-\frac{c}{d}$ covers $B_2$. Indeed, for any $x$ such that $\|x\|_2\ge 1/2$ there is a coordinate value $x_k$ such that $|x_k|\ge \frac{1}{2d^{1/2}}$. Suppose $x_k\ge \frac{1}{2d^{1/2}}$. Then
$$
x_k^2 - (x_k-a)^2= 2x_ka-a^2 \ge \frac{3}{16d}.
$$
This implies that the above explicitly written  union of  $2d$ balls of radius $r\ge
1-\frac{c}{d}$ covers $B_2$. We can use this covering for building an explicit covering with smaller $r$. The idea is to iterate  $m$ times the above covering with $r=1-\frac{c}{d}$. Then the radius of the resulting covering is $r=(1-\frac{c}{d})^m$ and the total number of balls in the covering does not exceed
$(2d)^m$. Using the notation $\delta:=1-r$ we get for small $\delta$
$$
(2d)^m \le \exp(Cd\delta \ln (2d)). 
$$

{\bf Example 6.} Let $X$ be a uniformly smooth Banach space $\R^d$ with norm $\|\cdot\|$ and a basis 
$\Psi:=\{\psi^j\}_{j=1}^d$. Then for any $x$ we have a unique representation 
$$
x=\sum_{j=1}^d x_j\psi^j,\quad |x_j|\le K\|x\|,\quad j=1,\dots,d.
$$
Let $\oo(u)$ be a continuous majorant of $\rho(u)$, $u\in[0,\infty)$, such that $\oo(u)/u$ monotone decreases  to $0$ as $u\to 0$. Set $a:=a(\oo,\frac{1}{Kd})$ to be a solution to the equation  (\ref{5.4}) with $\mu:=\frac{1}{Kd}$. 
Take a dictionary $\D:=\{\pm \psi^j\}_{j=1}^d$
and consider the covering
$$
\left(\cup_{j=1}^d B_X(a\psi^j,r)\right)\cup\left(\cup_{j=1}^d B_X(-a\psi^j,r)\right).
$$
We prove that  the above union with $r\ge 1-\frac{1}{2} a\mu$ covers $B_X$. We have
$$
\|x\| = F_x(x) = \sum_{j=1}^d x_jF_x(\psi^j) \le K\|x\|\sum_{j=1}^d |F_x(\psi^j)|,
$$
which implies that for some $k\in[1,d]$ 
\begin{equation}\label{6.8}
|F_x(\psi^k)| \ge (Kd)^{-1} =: \mu.
\end{equation} 
Suppose $F_x(\psi^k) \ge (Kd)^{-1}$. Then by Lemma \ref{L3.1} we get
$$
\|x-a\psi^k\| \le \|x\| -a  F_x(\psi^k)+2\|x\|\rho(2a)
$$
$$
 < 1-a\mu   +2\rho(2a) \le 1- \frac{1}{2}a\mu.
$$
This implies that the above explicitly written  union of  $2d$ balls of radius $r\ge
1-\frac{1}{2}a\mu$ covers $B_X$. We can use this covering for building an explicit covering with smaller $r$. As in Example 5 we iterate  $m$ times the above covering with $r=1-\frac{1}{2}a\mu$. Then the radius of the resulting covering is $r=(1-\frac{1}{2}a\mu)^m$ and the total number of balls in the covering does not exceed
$(2d)^m$.

Examples 5 and 6 demonstrate how simple constructions of coverings with $r<1$ can be used for $\e$-coverings. Suppose we can construct a $r$-covering with polynomial bound $N_r(B_X)\le d^{A_4}$ with $r=1-\frac{A_5}{d}$. Then, assuming that $X$ has smoothness of order $u^2$, as in Example 5, iterating this covering $m$ times we get a $(1-\delta)$-covering with
$$
N_{1-\delta}(B_X)\le \exp(Cd\delta \ln d),\quad \delta\ge \frac{A_5}{d}.
$$
This bound compared with the optimal bound from Proposition \ref{P1.1} contains an extra $\ln d$ factor in the exponent. However, a construction of an extremal 
dictionary for some fixed $\mu_0\ge c_0>0$ will give 
$$
N_{r_0}(B_X) \le \exp(C(c_0)d),\quad  r_0=r_0(c_0)<1.
$$
Iterating this construction we obtain
$$
N_\e(B_X)\le \exp(Cd\ln(1/\e))
$$
which is optimal in the sense of order of the exponent.


\newpage

\begin{thebibliography}{9999}

\bibitem{XL} Xiteng Liu, Sparse Signal Representation in Redundant Systems, Ph. D. Dissertation, University of South Carolina, 2006. 
 
\bibitem{LGM} G.G. Lorentz, M.v. Golitschek and Yu. Makovoz, Constructive Approximation. Advanced Problems, Springer, 1996.

\bibitem{T23}  V.N. Temlyakov, Greedy Approximations,  Foundations of Computational Mathematics, Santander 2005, {\em London Mathematical Society Lecture Notes Series}, {\bf 331} 2006, Cambridge University Press, 371--394. 

\bibitem{Tbook} V.N. Temlyakov, Greedy approximation, Cambridge University Press, 2011.

\bibitem{LW01}
J. van Lint and R. Wilson, \emph{A course in combinatorics}, second
edition, Cambidge Univeersity Press, 2001


\end{thebibliography}
\end{document}